\documentclass[11pt]{amsart}

\usepackage{amsmath,amsfonts,amssymb,latexsym}

\usepackage[kerning=true]{microtype} 
\usepackage{graphicx}
\usepackage{wrapfig}

\usepackage{stmaryrd}

\theoremstyle{theorem}

\newtheorem{theorem}{Theorem}

\newtheorem{lemma}{Lemma}

\newtheorem{corollary}{Corollary}

\newtheorem{claim}{Claim}

\theoremstyle{definition}

\newtheorem{definition}{Definition}

\newcommand{\sys}{{\rm sys}}

\newcommand{\R}{{\mathbb R}}

\newcommand{\Z}{{\mathbb Z}}

\newcommand{\A}{\mathcal A}

\numberwithin{equation}{section}
\numberwithin{theorem}{section}
\numberwithin{corollary}{section}
\numberwithin{definition}{section}
\numberwithin{lemma}{section}
\numberwithin{claim}{section}

\title{Measurements of Riemannian two-disks and two-spheres}

\author{F. Balacheff}

\address{F. Balacheff, Laboratoire Paul Painlev\'e, Bat. M2, Universit\'e des Sciences et Technologies,
59 655 Villeneuve d'Ascq, France.}

\email{florent.balacheff@math.univ-lille1.fr}



\begin{document}%

\begin{abstract}
We prove that any Riemannian two-sphere with area at most $1$ can be continuously mapped onto a tree in a such a way that the topology of fibers is controlled and their length is  less than $7.6$. This result improves previous estimates and relies on a similar statement for Riemannian two-disks. 
\end{abstract}

\maketitle

\section{Introduction}

In this article we are interested to describe the possible geometries of Riemannian two-disks and two-spheres in the same way a taylor determines the geometry of a body: by taking some relevant measurements. We denote by $A(\cdot)$ the area functionnal and $\vert \cdot \vert$ the length functionnal.
Our main result deals with measurements of two-disks and states as follows.

\begin{theorem}\label{th:maindisk}
If $D$ is a Riemannian two-disk, then for any $\epsilon >0$ we can find a continuous map to a trivalent tree such that the preimage of a terminal vertex is either an interior point or the boundary $\partial D$, the preimage of an interior point of an edge is homeomorphic to a circle, the preimage of a trivalent vertex is homeomorphic to the $\theta$ figure,  and fibers have length at most
$$
(1+\epsilon) \, \max \left\{\{\vert \partial D \vert +\sqrt{A(D)}, \left(4+{11\sqrt{3}\over 4}\right)\, \sqrt{A(D)}\right\}.
$$
\end{theorem}

This theorem should be compared to a result by Y. Liokumovich in \cite{Li13} which states that any Riemannian two-disk $D$ admits a Morse function $f :D \to \R$ which is constant on the boundary and whose fibers have length at most $52\sqrt{A(D)}+\vert \partial D\vert$.

\medskip

Using Theorem \ref{th:maindisk}, we are able to estimate the measurements of two-spheres in terms of their area.

\begin{theorem}\label{th:mainsphere}
If $M$ is a Riemannian two-sphere with area less than  $1$, then it admits a continuous map to a trivalent tree such that the preimage of a terminal vertex is a point, the preimage of an interior point of an edge is homeomorphic to a circle, the preimage of a trivalent vertex is homeomorphic to the $\theta$ figure and fibers have length at most $2\sqrt{3}+33/8 \simeq 7.6$.
\end{theorem}

This improves a previous estimate by Y. Liokumovich proving such a result with $8\sqrt{3}+12\simeq 26$ as upperbound on the length of the fibers, see \cite{Li13}.  Note nevertheless that the main result of Y. Liokumovich is stronger: he proves the existence of a Morse function $f:M\to \R$ whose fibers have length at most $52$. Also remark that L. Guth proved in \cite{Gu05} the existence of maps such as in Theorem \ref{th:mainsphere} with the upperbound $120 / (2\sqrt{\pi})\simeq 34$ on the length of the fibers under the weaker assumption that the $1$-hypersphericity is less than $1/(2\sqrt{\pi})$.\\

The interest of obtaining precise measurements of two-spheres is illustrated in the fact that we can derive upperbounds on the shortest length of a closed geodesic, on the shortest length of a simple loop dividing the sphere into two subdisks of area at least $A/3$, or on the maximal length of a shortest pants decomposition for punctured spheres. More precisely, we are thus able to recover the result of C. Croke in \cite{Cr88} on the existence of short closed geodesics for Riemannian two-spheres. More precisely we can deduce from Theorem \ref{th:mainsphere} that any Riemannian two-sphere with unit area carries a closed geodesic of length at most $\simeq 10.1$, see Theorem \ref{th:sys}. This is not as good as the actual best constant---which is $4\sqrt{2}\simeq 5.7$ and is due to R. Rotman (see \cite{Ro05})---but this is not too far. Moreover, using Theorem \ref{th:mainsphere}, we can also recover Theorem VI of \cite{ABT13} on the existence of a short closed geodesic  for Finsler (eventually non-reversible) two-spheres. More precisely, we have the following.

\begin{theorem}\label{th:mainFinsler}
Let $M$ be a Finsler (eventually non-reversible) two-sphere with Holmes-Thompson area less than $1$. Then it carries a closed geodesic of length at most $2\sqrt{\pi} \, (11\sqrt{3}+16)\simeq 31.1$.
\end{theorem}

This result improves the actual best constant due to Y. Liokumovich, see \cite[Theorem 4]{Li13}. 
We can also easily deduce from Theorem \ref{th:mainsphere} the following result.

\begin{theorem}
Let $M$ be a Riemannian two-sphere.
Then there exists a simple loop of length at most $(2\sqrt{3}+33/8) \, \sqrt{A(M)}$ dividing $M$ into two subdisks of area at least $A(M)/3$.
\end{theorem}

This result improves one by Y. Liokumovich, see \cite[Theorem 1]{Li13}.
Finally it is straightforward to see that Theorem \ref{th:mainsphere} implies the following.

\begin{theorem}
Let $M$ be a Riemannian punctured two-sphere with area less than $1$.
Then there exists a decomposition of $M$ into $3$-holed spheres such that each boundary curve has length at most $2\sqrt{3}+33/8\simeq 7.6$.
\end{theorem}

It improves the actual best bound, even for hyperbolic metrics, compare with \cite{BP12}.\\

The paper is organized as follows. The first section presents Besicovich's lemma, and some of its useful corollaries : Papasoglu's lemma which first appeared in \cite{Pap10}, and the disk subdivision lemma which can be found in \cite{LNR12}. We also define an invariant called the $\theta$-width, reformulate Theorem \ref{th:maindisk} in terms of this invariant and show how to prove Theorem \ref{th:mainsphere} and Theorem \ref{th:mainFinsler} from Theorem \ref{th:maindisk}. In the second section, we prove Theorem \ref{th:maindisk}. Our strategy is inspired by the proof of  \cite[Theorem 1.6]{LNR12} where Y. Liokumovich, A. Nabutovsky and R. Rotman show that the boundary of any Riemannian two-disk with uniformly bounded diameter and area can always be contracted through closed curves of bounded length. We show that it is enough to consider the case where the length of the boundary is short in comparison with the area. This step is done by using Besicovich's lemma. Then we  use the disk subdivision lemma to argue by induction on the area. \\

\noindent {\bf Acknowledgements.} The author gratefully thanks Y. Liokumovich, A. Nabutovsky, R. Rotman and K. Tzanev for valuable discussions.
The author acknowledges partial support of this work by the Programme ANR FINSLER.
 The paper was mainly written during a visit at the CRM of Barcelona. The author would like to thank this institute for its kind hospitality.

\section{Preliminaries}

As we deal only with surfaces, we will use the terms of disk and sphere for two-disk and two-sphere. We denote by $A(\cdot)$ the area functionnal and $\vert \cdot \vert$ the length functionnal.

\subsection{Besicovich's lemma and consequences}

In order to prove our results, we will use the following fundamental result of metric geometry as well as some of its consequences.

\begin{lemma}[Besicovich \cite{Bes52}]\label{lem:besi}
Let ${\mathcal S}$ be a Riemannian square. Then there exists a simple geodesic path connecting two opposites sides of length at most $\sqrt{A({\mathcal S})}$.
\end{lemma}

In particular, any Riemannian disk $D$ whose boundary satisfies $\vert \partial D \vert > 4 \sqrt{A(D)}$ can be subdivided into two subdisks of smaller perimeters (divide its boundary into four equal parts and apply Besicovich's lemma).\\

In \cite{Pap10}, P. Papasoglu used Besicovich's lemma to derive the following estimate.

\begin{lemma}[Papasoglu]\label{lem:papa1}
Let $M$ be a Riemannian two-sphere.
For any $\delta >0$ there exists a simple closed curve of length at most $2\sqrt{3}\sqrt{A(M)}+\delta$ and subdviding $M$ into two disks of area at least $A(M)/4$.
\end{lemma}

In \cite[Proposition 3.2]{LNR12}, the authors apply Papasoglu's result to cut Riemannian disks into two parts of sufficiently big area by a curve of controlled length. We reformulate their result as follows. 

\begin{lemma}\label{lem:papa2}
Let $D$ be a Riemannian two-disk.
For any $\lambda<{1 \over 4}$ and $\delta >0$ there exists a subdisk $D' \subset D$ such that 
$\lambda A(D)\leq A(D')\leq (1-\lambda) A(D)$ and
$\vert \partial D' \setminus \partial D \vert \leq2\sqrt{3}\sqrt{A(D)}+\delta$.
\end{lemma}

\bigskip

\subsection{Technical width}

In this section we introduce our main tool, the $\theta$-width, reformulate Theorem \ref{th:maindisk} in terms of this invariant and show how to derive Theorem \ref{th:mainsphere} and Theorem \ref{th:mainFinsler}.
 
\begin{definition}[$\theta$-width]
Let $M$ be a compact Riemannian surface (possibly with non-empty boundary).
We define the $\theta$-width denoted by $W_\theta(M)$ as the infimum of the $L>0$ such that there exists  a continuous map $f$ from $M$ to a trivalent tree $T$ satisfying the following conditions: 
\begin{enumerate}
\item[(W.1)] $f(\partial M) \subset \partial T$ and the preimage of a terminal vertex is either an interior point or a connected component of $\partial M$,
\item[(W.2)] the preimage of an interior point of an edge is homeomorphic to a circle,
\item[(W.3)] the preimage of a trivalent vertex is homeomorphic to the letter $\theta$,
\item [(W.4)] the preimage of any point has length at most $L$.
\end{enumerate}
\end{definition}

Observe that in particular $W_\theta(M)\geq \vert \partial M\vert$.\\

Our results are consequences of the following estimate.

\begin{theorem}\label{th:disk}
Let $D$ be a Riemannian two-disk.
Then
$$
W_\theta(D) \leq \max \left\{\vert \partial D \vert +\sqrt{A(D)}, \left(4+{11\sqrt{3}\over 4}\right) \, \sqrt{A(D)}\right\}.
$$
\end{theorem}

We will prove this theorem in section \ref{sec:disk}. This is straightforward to check that it implies Theorem \ref{th:maindisk}. Observe that it also implies the following statement, of which Theorem \ref{th:mainsphere} is  a direct consequence.

\begin{corollary}\label{th:sphere}
Let $M$ be a Riemannian two-sphere. Then 
$$
W_\theta(M)\leq \left(2\sqrt{3}+{33 \over 8}\right) \, \sqrt{A(M)}.
$$
\end{corollary}

\begin{proof}[Proof of Corollary \ref{th:sphere}]
 Let $M$ be a Riemannian two-sphere. 
First we use Papasoglu's result (Lemma \ref{lem:papa1}) to divide $M$ into two disks $D_1$ and $D_2$ of area at least $A(M)/4$ by a simple close curve of length at most $3\sqrt{3}\sqrt{A(M)}$. Now for each subdisk we have the following bound according to Theorem \ref{th:disk}:
\begin{eqnarray*}
W_\theta(D_i)& \leq &\max \left\{\vert \partial D_i \vert+\sqrt{A(D_i)} , \left(4+{11\sqrt{3}\over 4}\right) \, \sqrt{A(D_i)} \right\}\\
& \leq& \left(2\sqrt{3}+{33 \over 8}\right) \, \sqrt{A(M)}
\end{eqnarray*}
as $A(D_i) \leq {3 \over 4} \, A(M)$ for $i=1,2$ and $\vert \partial D_1 \vert = \vert \partial D_2 \vert \leq 3\, \sqrt{3}\sqrt{A(M)}$.
It is straightforward to check that
$$
W_\theta(M)\leq \max \left\{W_\theta(D_1),W_\theta(D_2)\right\}\leq  \left(2\sqrt{3}+{33 \over 8}\right) \, \sqrt{A(M)}.
$$
\end{proof}

\subsection{Existence of short closed geodesics.} It is classical to derive for spheres the existence of short closed geodesics from bounds on the $\theta$-width. In particular,

\begin{theorem}\label{th:sys}\hfill

\medskip

\noindent 1) Let $M$ be a Riemannian two-sphere with area $1$. 
Then it carries a closed geodesic of length at most $8/\sqrt{3}+11/2\simeq 10.1$.\\

\noindent 2) Let $M$ be a Finsler reversible two-sphere with Holmes-Thompson area $1$.
Then it carries a closed geodesic of length at most $\sqrt{\pi/2} \, (8/\sqrt{3}+11/2)\simeq 12.7$.\\

\noindent 3) Let $M$ be a Finsler (eventually non-reversible) two-sphere with Holmes-Thompson area $1$. 
Then it carries a closed geodesic of length at most $\sqrt{3\pi} \, (8/\sqrt{3}+11/2)\simeq 31.1$.
\end{theorem}

\begin{proof} It follows from \cite[section 4.4]{ABT13} that 
\begin{itemize} 
\item if any Riemannian sphere $M$ with unit area satisfies $W_\theta(M)\leq C$, then any reversible Finsler sphere $M'$ with unit Holmes-Thompson area satisfies $W_\theta(M')\leq \sqrt{\pi /2} \, C$;
\item if any reversible Finsler sphere $M$ with unit Holmes-Thompson area satisfies $W_\theta(M)\leq C$, then any Finsler sphere $M''$ with unit Holmes-Thompson area satisfies $W_\theta(M')\leq \sqrt{6} \,C$.
\end{itemize}
Now fix a Finsler sphere $M$. We denote by $\sys(M)$ the systole of $M$ defined as the length of a shortest closed geodesic.
According to Corollary \ref{th:sphere} it remains to prove that 
$$
\sys(M)\leq {4\over3} \, W_\theta(M).
$$
\\

The existence of a closed geodesic on $M$ can be proved through a minimax argument on the one-cycle space~$\mathcal{Z}_{1}(M;\Z)$. We refer the reader to \cite{BS10} and the references therein for additional information. Loosely speaking, this space arising from geometric measure theory is made of multiple curves (unions of oriented loops) endowed with some special topology.  This space permits to define a minimax process on the Finsler sphere~$(M,F)$ using F.~Almgren's isomorphism between the relative fundamental group $\pi_{1}(\mathcal{Z}_{1}(M;\Z),\{0\})$ and the second homology group $H_{2}(M;\Z)\simeq \Z$. From a result of J.~Pitts, the minimax quantity
$$
 \inf_{(z_{t})} \sup_{0 \leq t \leq 1} \vert z_t\vert
$$
where $(z_{t})$ runs over the families of one-cycles inducing a non-trivial element of \linebreak $\pi_{1}(\mathcal{Z}_{1}(M;\Z),\{0\})$ bounds from above the systole.
\\

We argue by contradiction. Suppose that $\sys(M)> {4 \over 3} \,W_\theta(M)$. Fix $\epsilon>0$ such that $\sys(M)>{4 \over 3} \,W_\theta(M)+\epsilon$. By definition there exists a continuous map $f$ from $M$ to a trivalent tree $T$ satisfying (W.1-4) with length $L=W_\theta(M)+\epsilon$. 

Let $v$ be a trivalent vertex. Its preimage denoted by $\theta(v)$ is made of  three disjoint oriented arcs $\alpha_1$, $\alpha_2$ and $\alpha_3$  with the same endpoints ordered such that $\vert \alpha_1 \vert \leq \vert \alpha_2 \vert  \leq \vert \alpha_3 \vert$. Denote by $\beta_{ij}$ for $1\leq i < j\leq 3$ the concatenation of the oriented arcs $\alpha_i$ with $-\alpha_j$. 
As $\sys(M)>\vert \beta_{ij}\vert$
we can continuously contract each of the $\beta_{ij}$ to a point curve through a homotopy by using a Birkhoff  process, see \cite{Cr88}. We denote by $\{\beta_{ij}^t\}_{t \in [0,1]}$ this homotopy with the convention that $\beta_{ij}^0=\beta_{ij}$. We define an element  of $\pi_{1}(\mathcal{Z}_{1}(M;\Z),\{0\})$ by
$$
f_v: t\in [0,1] \to \left\{\begin{array}{cc}
      &    -\beta_{12}^{1-2t} + \beta_{13}^{1-2t} \, \text{for} \, t \in [0,1/2];\\
      & \\
      &   \beta_{23}^{2t-1}  \hfill \, \text{for} \, t \in [1/2,1].
\end{array}\right.
$$
This gives rise to an element $[f_v] \in H_2(M,\Z)$  such that $\vert f_v(t)\vert \leq {4 \over 3} \, W_\theta(M)+\epsilon$ for any $t \in [0,1]$.

Now fix an edge $e=[v_0,v_1]\simeq [0,1]$ which is not terminal. We denote by $\alpha_t$ the preimage of an interior point of $e$ corresponding to the parameter $t \in ]0,1[$ and orient them in a coherent way. For $i=0,1$ denote by $\alpha_i$ the oriented  curve obtained as the limit of the curves $\alpha_t$ when $t \to i$.
The curve $\alpha_i$ is a simple closed curve contained in $\theta(v_i)$. As before we can contract $\alpha_i$ to a point through an homotopy $\{\alpha_i^t\}_{t \in [0,1]}$. 
We define an element  of $\pi_{1}(\mathcal{Z}_{1}(M;\Z),\{0\})$ by
$$
f_e: t\in [0,1] \to \left\{\begin{array}{cc}
      &    \alpha_0^{1-3t} \hfill  \text{for} \, t \in [0,1/3];\\
      & \\
      &    \alpha_{3t-1} \hfill  \, \text{for} \, t \in [1/3,2/3];\\
      & \\
      &   \alpha_1^{3t-2}  \hfill \, \text{for} \, t \in [2/3,1].
\end{array}\right.
$$
This gives rise to an element $[f_e] \in H_2(M,\Z)$  such that $\vert f_e(t)\vert \leq W_\theta(M)+\epsilon$  for any $t \in [0,1]$.

Finally fix a terminal edge $e=[v_0,v_1]\simeq [0,1]$ with the terminal vertex corresponding to $0$. With the same notations as above, the curve $\alpha_0$ is reduced to a point curve. We define an element  of $\pi_{1}(\mathcal{Z}_{1}(M;\Z),\{0\})$ by
$$
f_e: t\in [0,1] \to \left\{\begin{array}{cc}
      &  \alpha_{2t} \hfill  \, \text{for} \, t \in [0,1/2];\\
      & \\
      &    \alpha_1^{2t-1} \hfill  \, \text{for} \, t \in [1/2,1].
\end{array}\right.
$$
This gives rise to an element $[f_e] \in H_2,(M,\Z)$ such that $\vert f_e(t)\vert \leq W_\theta(M)+\epsilon$ for any $t \in [0,1]$.\\

It is possible to choose all the orientations in such a way that
$$
[S^2]=\sum_{e \in E(T)} [f_e]+\sum_{v \in V(T)} [f_v].
$$ 
Here $E(T)$ and $V(T)$ denote respectively the set of edges and the set of vertices of $T$.
It implies that there exists an edge $e$ such that $[f_e]\neq 0$ or a vertex $v$ such that $[f_v]\neq 0$. According to the minimax principle on the one-cycle space, we conclude that
$$
\sys(M) \leq {4 \over 3} \, W_\theta(M)+\epsilon
$$
which is a contradiction.
\end{proof}

\section{$\theta$-width of a Riemannian disk}\label{sec:disk}

In this section we prove Theorem \ref{th:disk}. For this we adapt the strategy of the proof of \cite[Theorem 1.6]{LNR12} to control our invariant $W_\theta$.

\subsection{Reduction to the short boundary case}

\begin{lemma}\label{lem:short}
Let $D$ be a Riemannian two-disk and $C\geq0$. Suppose that there exists $\eta>0$ such that for any subdisk $D' \subset D$ for which 
$$
\vert \partial D' \vert <  (4+\eta) \, \sqrt{A(D')}
$$ 
we have 
$$
W_\theta(D')\leq (1+\eta) \, \max\left\{ \vert \partial D'\vert+\sqrt{A(D')}, C \, \sqrt{A(D')} \right\}.
$$
 Then for any subdisk $D'\subset D$ we have that 
$$
W_\theta(D')\leq (1+\eta) \, \max \left\{\vert \partial D'\vert+\sqrt{A(D')} ,C \, \sqrt{A(D')}\right\}.
$$ 
\end{lemma}

We will use in the sequel this lemma with the constant $C=0$ (small area case) and $C=C_{\lambda,\eta}:=4+2\eta +2\sqrt{3}+{1-\lambda\over \sqrt{3}(1-2\eta)}+\sqrt{1-\lambda}$ for $0<\lambda<{1\over 4}$ and $\eta >0$ (general case).

\begin{proof}
For any subdisk $D' \subset D$ we define $n(D')$ to be the smaller integer $n$ such that
$$
\vert \partial D'\vert < \left(4+\eta\, \left({4 \over 3}\right)^n\right)\sqrt{A(D')}.
$$

Let $D'$ be a subdisk such that $n(D')= 0$. Equivalently we have that $\vert \partial D' \vert < (4+\eta)\sqrt{A(D')}$ and so we are done by assumption.\\

Now fix an integer $n$ and suppose that for any subdisk $D' \subset D$ such that $n(D')\leq n-1$ we have proven that   
$$
W_\theta(D')\leq (1+\eta) \, \max \left\{ \vert \partial D'\vert +\sqrt{A(D')}, C\, \sqrt{A(D')}\right\}.
$$
Let $D'\subset D$ be a subdisk with $n(D')= n$. In particular $\vert \partial D' \vert > 4 \sqrt{A(D')}$ so we can subdivide $D'$ into two subdisks $D'_1$ and $D'_2$ of smaller perimeters using a Besicovich's cut $\alpha$ of length $\sqrt{A(D')}$ (Lemma \ref{lem:besi}). More precisely,

\begin{eqnarray*}
\vert \partial D'_i \vert & \leq &{3\over 4} \vert \partial D' \vert +\sqrt{A(D')}\\
& < & {3 \over 4} \left(\eta\, \left({4 \over 3}\right)^n+4\right)\sqrt{A(D')}+\sqrt{A(D')}\\
& < & \left(\eta\, \left({4 \over 3}\right)^{n-1}+4\right)\sqrt{A(D')}
\end{eqnarray*}
so $n(D'_i)\leq n-1$.\\

Let $\epsilon>0$ be small enough so that all points of $D'$ at a distance at least $\epsilon$ from $\partial D'$ form a subdisk denoted by $D''\subset D'$. The subdisk $D''$ is itself subdivided by the Besicovich's cut $\alpha$ into two subdisks $D''_i \subset D'_i$ for $i=1,2$. By considering $\epsilon$ smaller if necessary, we can suppose that $\partial D''_i$ is sufficiently closed to $\partial D'_i$ so that $\vert \partial D''_i\vert < \vert \partial D' \vert$ and $n(D''_i)\leq n-1$. In particular  $W_\theta(D''_i)\leq (1+\eta) \, \max \left\{\vert \partial D''_i\vert +\sqrt{A(D''_i)} , C\, \sqrt{A(D''_i)}\right\}$ for $i=1,2$ according to the induction assumption, which implies that
$$
W_\theta(D''_i)\leq (1+\eta) \, \max \left\{\vert \partial D'\vert +\sqrt{A(D')} , C\, \sqrt{A(D')}\right\}.
$$

\begin{claim}\label{cl:1}
$$
W_\theta(D')\leq \max \left\{\vert \partial D'\vert+\sqrt{A(D')}+o(\epsilon), W_\theta(D''_1), W_\theta(D''_2)\right\}.
$$
\end{claim}

\noindent $\rhd$ Indeed for any $\delta>0$ and $i=1,2$ let $f_i : D''_i \to T_i$ be a continuous map to a trivalent tree $T_i$ satisfying conditions (W.1-4) with length strictly less than $W_\theta(D''_i)+\delta$. Denote by $v_i$ the terminal vertex of $T_i$ corresponding to the boundary $\partial D''_i$. Consider a new edge $e \simeq [0,1]$ and define a new trivalent tree $T$ obtained from $T_1$, $T_2$ and $e$  by identifiying $v_1$, $v_2$ and the vertex of $e$ corresponding to $\{1\}$ into the same vertex denoted by $v$. The trees $T_1$ and $T_2$ can be thought as subgraphs of $T$. Denote by $\{\gamma_t\}_{t \in [0,1]}$ a monotone isotopy from $\partial D'$ to $\partial D''$ form by level sets of the distance function to $\partial D'$. It satisfies that
$$
\vert \gamma_t \vert \leq \vert \partial D'\vert +o(\epsilon).
$$

We define a new map $f:D' \to T$ as follows:
\begin{itemize}
\item $f(x)=f_i(x)$ if $x \in D''_i\setminus \partial D''_i$;
\item $f(x)=v$ if $x \in \partial D''_1 \cup \partial D''_2 $;
\item $f(x)=t$ if $x \in \gamma_t$ for $t \in [0,1]$.  
\end{itemize}
By construction we have that the length of the preimages is always strictly less than  
$$
\max\left\{\vert \partial D'\vert+\sqrt{A(D')}+o(\epsilon),W_\theta(D''_1)+\delta,W_\theta(D''_2)+\delta\right\}.
$$
It is easy to check that $f:D' \to T$ satisfies the conditions (W.1-3) which concludes the proof by letting $\delta \to 0$. $\lhd$\\

It implies that 
$$
W_\theta(D')\leq (1+\eta) \, \max \left\{ \vert \partial D'\vert +\sqrt{A(D')} , C\, \sqrt{A(D')}\right\}
$$
by letting $\epsilon \to 0$ and we are done by induction.
\end{proof}

\subsection{Small area case.}

\begin{lemma}\label{lem:small}
Let $D$ be a Riemannian two-disk and $\eta>0$. There exists $\epsilon>0$ such that any subdisk $D' \subset D$ with $A(D')\leq \epsilon$ satisfies
$$
W_\theta(D')\leq (1+\eta) \, (\vert \partial D'\vert + \sqrt{A(D')}).
$$
\end{lemma}

\begin{proof}
According to Lemma \ref{lem:short} with $C=0$ it is enough to prove the lemma for subdisks $D'$ with 
$$
\vert \partial D' \vert < (4+\eta) \, \sqrt{\epsilon}.
$$
As observed in the proof of \cite[Lemma 2.3]{LNR12}, for $r$ small enough every ball of radius $r$ is $(1+O(r))$-bilipschitz homeomorphic to a convex subset of $\R^2$. Hence for $\epsilon$ small enough the condition $\vert \partial D' \vert < (4+\eta) \, \sqrt{\epsilon}$ ensures that $D'$ is $(1+\eta)$-bilipschitz to a subset $U \subset \R^2$ with smooth boundary. It is easy to continuously contract the boundary of $U$ into a point through a continuous one-parameter family of closed curves of $U$ with decreasing length. For this consider a supporting line $\ell$ of $U$. We linearly translate this line in the inner orthogonal direction until we sweep out $U$ and denote by $\left\{\ell_t\right\}_{t \in [0,1]}$ this family of translated lines (with the convention that $\ell_0=\ell$). For each $t \in [0,1]$ the intersection $\ell_t \cap U$ consists of a finite number of disjoint segments. By transversality we can assume moreover that this number of disjoint segments changes at each step by at most $1$. Consider the family of closed curves $\gamma_t$ defined as the boundary of the union $\cup _{s \in [t,1]} U \cap \ell_t$. This is a continuous one-parameter family of closed curves (with eventually multiple connected components) of $U$ with decreasing length that contracts $\partial U$ to a point. The curves involved in this family are not disjoint, but it can be done by slightly pertubing the family in the neighborhood of $\partial U$ without significantly increasing their length. Finally it is classic to derive from this family a map $f : U \to T$ with $T$ a trivalent tree and satisfying conditions (W.1-4) with $L$ as close as wanted from $\vert \partial U\vert$, compare with \cite[p.128]{Gr83}.
In particular $W_\theta(U) \leq \vert \partial U \vert$ which in turn implies that $W_\theta(D') \leq (1+\eta) \, \vert \partial D' \vert$.
\end{proof}

\subsection{General case.} Let $D$ be a Riemannian disk. Fix $\eta >0$ and $0<\lambda<{1 \over 4}$ and define
$$
C_{\lambda,\eta}=4+2\eta +2\sqrt{3}+{1-\lambda\over \sqrt{3}(1-2\eta)}+\sqrt{1-\lambda}.
$$ 

We will argue by induction and prove that for any subdisk $D' \subset D$ we have
$$
W_\theta(D') \leq (1+\eta) \, \max\left\{ \vert \partial D' \vert+ \sqrt{A(D')}, C_{\lambda,\eta} \, \sqrt{A(D')} \right\}.
$$
In particular it implies the conclusion of Theorem \ref{th:disk} by letting $\eta \to 0$ and $\lambda \to {1 \over 4}$.  According to lemma \ref{lem:short} with $C=C_{\lambda,\eta}$ it is enough to estimate the $\theta$-width of $D'$ under the stronger assumption that $\vert \partial D' \vert < (4+\eta) \, \sqrt{A(D')}$.\\

Let $\epsilon>0$ such that the conclusion of lemma \ref{lem:small} holds. For any subdisk $D' \subset D$ we define $m(D')$ to be the smaller integer $m$ such that 
$$
A(D') \leq\epsilon \, \left( {1 \over 1-\lambda} \right)^m.
$$
\\

\noindent Let $D'$ be a subdisk such that $m(D')= 0$. Equivalently  $A(D') \leq \epsilon$ and we are done according to lemma \ref{lem:small}.\\

\noindent Now fix a positive integer $m$ and suppose that for any subdisk $D' \subset D$ with $m(D')\leq m-1$ we have  proven that 
$$
W_\theta(D')\leq (1+\eta) \, \max \left\{ \vert \partial D' \vert+ \sqrt{A(D')}, C_{\lambda,\eta} \, \sqrt{A(D')}\right\}.
$$  
 Let $D'\subset D$ be a subdisk with $m(D')=m$.

By Lemma \ref{lem:papa2} there exists a subdisk $D'_0 \subset D'$ such that 
\begin{itemize}
\item $\lambda  A(D')\leq A(D'_0)\leq (1-\lambda) A(D')$,
\item $\vert \partial D'_0 \setminus \partial D' \vert \leq (2\sqrt{3}+\eta) \, \sqrt{A(D')}$.
\end{itemize}

\bigskip

\noindent {\it First case.} If $\partial D'_0 \cap \partial D'\neq \emptyset$, then $D'$ decomposes into an union of subdisks $D'_0,\ldots,D'_k$ with disjoint interiors such that $A(D'_i)\leq (1-\lambda) A(D')$ for $i=0,\ldots,k$. In particular for each $i$ we have 
$$
W_\theta(D'_i)\leq (1+\eta) \, \max \left\{\vert \partial D'_i\vert +\sqrt{A(D'_i)} , C_{\lambda,\eta} \,   \sqrt{A(D'_i)}\right\}
$$
as $m(D'_i)\leq m-1$ and the inductive assumption applies.

Using a similar argument to that of Claim \ref{cl:1}, it is straightforward to check that
$$
W_\theta(D') \leq (1+\eta) \max \left\{\vert \partial D'\vert + \vert \partial D'_0 \setminus \partial D' \vert+\sqrt{1-\lambda}\, \sqrt{A(D')}, C_{\lambda,\eta} \,   \sqrt{1-\lambda}\sqrt{A(D')}\right\}
$$
as $\vert \partial D'_i \vert \leq \vert \partial D'\vert + \vert \partial D'_0 \setminus \partial D' \vert$ and $A(D'_i)\leq (1-\lambda) A(D')$  for $i=0,\ldots,k$.

It implies that
\begin{eqnarray*}
W_\theta(D') &\leq&(1+\eta)\max \left\{(4+2\eta + 2\sqrt{3}+\sqrt{1-\lambda}) \, \sqrt{A(D')},C_{\lambda,\eta}\, \sqrt{A(D')}\right\}\\
&\leq&(1+\eta)\max \left\{\vert \partial D'\vert + \sqrt{A(D')} ,C_{\lambda,\eta} \, \sqrt{A(D')}\right\}
\end{eqnarray*}
as claimed.

\bigskip

\noindent {\it Second case.} If $\partial D'_0 \cap \partial D'= \emptyset$, then $D'$ decomposes into the union of the disk $D'_0$ and an annulus ${\mathcal A}$. Recall that
\begin{itemize}
\item $\vert \partial D' \vert < (4+\eta) \, \sqrt{A(D')},$
\item $\vert \partial D'_0 \vert \leq (2\sqrt{3}+\eta) \, \sqrt{A(D')},$
\item $A({\mathcal A})\leq (1-\lambda) A(D')$,
\item $A(D'_0)\leq (1-\lambda) A(D').$
\end{itemize}
In particular $m(D'_0)\leq m-1$ so that 
\begin{eqnarray*}
W_\theta(D'_0)&\leq& \max \left\{ \vert \partial D'_0\vert +\, \sqrt{A(D'_0)}, C_{\lambda,\eta} \, \sqrt{A(D')}\right\}\\
& \leq&   C_{\lambda,\eta} \, \sqrt{A(D')}.
\end{eqnarray*}
by the inductive assumption.\\

Denote by $h({\mathcal A})$ the height of the annulus, that is the distance between its two boundary curves. We say that ${\mathcal A}$ decomposes into a stack of annuli if there exist a finite number of annuli ${\mathcal A}_1,\ldots, {\mathcal A}_k$ with disjoint interiors such that ${\mathcal A}=\cup_{i=1}^k {\mathcal A}_i$ and  ${\mathcal A_i}\cap{\mathcal A_{i+1}}=\beta_i$ is a common boundary simple closed curve for $i=1,\ldots,k-1$. The following lemma will help us to estimate the $\theta$-width of $D'$.

\begin{lemma}
The Riemannian annulus ${\mathcal A}$ decomposes into a stack of annuli ${\mathcal A}_1,\ldots, {\mathcal A}_k$ such that 
$$
h(\A_i) \leq {\sqrt{1-\lambda}\over 2\sqrt{3}(1-2\eta)} \, \sqrt{A({\mathcal A})}
$$
for $i=1,\ldots,k$ and
$$
\vert \beta_i \vert \leq {2\sqrt{3} +\eta \over \sqrt{1-\lambda}} \, \sqrt{A({\mathcal A})}
$$
for $i=1,\ldots,k-1$.
\end{lemma}

\begin{proof}
Suppose that 
$$
h({\mathcal A})>  {\sqrt{1-\lambda}\over 2\sqrt{3}(1-2\eta)} \, \sqrt{A({\mathcal A})}.
$$ 
Consider for every $0<t<h({\mathcal A})$ the $1$-cycle $c_t$ formed by points of ${\mathcal A}$ at distance $t$ of $\beta_0$. By the coarea formula
$$
A(\left\{c_t \mid t\in [\eta h({\mathcal A}),(1-\eta)h({\mathcal A})]\right\} =  \int_{\eta h({\mathcal A})}^{(1-\eta)h({\mathcal A})} \vert c_t \vert dt
\leq A({\mathcal A})
$$
so that there exists some $t \in [\eta h({\mathcal A}),(1-\eta)h({\mathcal A})]$ such that
$$
\vert c_t \vert \leq {2\sqrt{3} \over \sqrt{1-\lambda}} \, \sqrt{A({\mathcal A})}
$$
otherwise we derive a contradiction. The cycle $c_t$ can be approximated by a union of smooth closed simple curves with total length at most 
$$
{2\sqrt{3} +\eta \over \sqrt{1-\lambda}} \, \sqrt{A({\mathcal A})}.
$$
So ${\mathcal A}$ decomposes into a stack of two annuli ${\mathcal A}_1$ and ${\mathcal A}_2$ such that ${\mathcal A}_1 \cap {\mathcal A}_2$ is a simple closed curve of length at most $((2\sqrt{3}+\eta)/\sqrt{1-\lambda})  \,A({\mathcal A})$ and such that $h({\mathcal A}_i)\leq (1-\eta) h({\mathcal A})$ for $i=1,2$. By iterating this process we derive the Lemma.
\end{proof}

We suppose that the stack decomposition is ordered in such a way that $D'_0$ and ${\mathcal A}_1$ are adjacent.  In the sequel we will denote by $\beta_0$ the boundary curve of ${\mathcal A}$ corresponding to $\partial D'_0$ and by $\beta_k$ the one corresponding to $\partial D'$. Observe in particular that
$$
\vert \beta_i\vert \leq (2\sqrt{3}+\eta) \sqrt{A(D')}
$$
for $i=0,\ldots,k-1$ and that 
$$
\vert \beta_k\vert \leq (4+\eta) \sqrt{A(D')}.
$$

Now it remains to estimate the $\theta$-width of $D'$ using this stack decomposition.
For each annulus ${\mathcal A}_i$ of the decomposition choose a minimizing simple path $\alpha_i$ between its two boundary curves. Cutting then along the curve $\alpha_i$ yields to a disk we denote by $D'_i$ whose boundary consists  in the concatenation of $\beta_{i-1}$, a copy of $\alpha_i$, $\beta_i$ and another copy of $\alpha_i$.  Observe that
\begin{eqnarray*}
\vert \partial D'_i \vert&\leq &(4+\eta)\sqrt{A(D')} + (2\sqrt{3}+\eta)\sqrt{A(D')}+2\left({\sqrt{1-\lambda}\over 2\sqrt{3}(1-2\eta)}\right)\sqrt{A({\mathcal A}_i)}\\
&\leq &\left(4+2\eta +2\sqrt{3}+{1-\lambda\over \sqrt{3}(1-2\eta)}\right) \sqrt{A(D')}.
\end{eqnarray*}
As $A(D'_i)=A({\mathcal A}_i)\leq (1-\lambda) A(D')$ , we have $m(D'_i)\leq n-1$ for $i=1,\ldots,k$ so that 
\begin{eqnarray*}
W_\theta(D'_i)&\leq& \max \left\{ \vert \partial D'_i \vert +\, \sqrt{A(D'_i)}, C_{\lambda,\eta} \, \sqrt{A(D'_i)}\right\}\\
& \leq&  \max \left\{\left(4+2\eta +2\sqrt{3}+{1-\lambda\over \sqrt{3}(1-2\eta)}+\sqrt{1-\lambda}\right) \sqrt{A(D')},C_{\lambda,\eta}\,\sqrt{A(D')}\right\}\\
&\leq & C_{\lambda,\eta} \, \sqrt{A(D')}.
\end{eqnarray*}
by the inductive assumption. 

\begin{lemma}
For $i=1,\ldots,k$,
$$
W_\theta(D'_0 \cup \ldots \cup D'_i)\leq \max \left\{W_\theta(D'_0\cup\ldots \cup D'_{i-1}),W_\theta(D'_i)\right\}.
$$
\end{lemma}

In particular 
$$
W_\theta(D')\leq \max \left\{W_\theta(D'_0),\ldots, W_\theta(D'_{k})\right\}\leq C_{\lambda,\eta}\, \sqrt{A(D)}
$$
which concludes the proof of Theorem \ref{th:disk}.

\begin{proof}
Fix $\delta>0$ and  $i \in \llbracket 1,k\rrbracket$. We can find a trivalent tree $T_i$ (resp. $T'_i$) together with a continuous map  $f_i:  D'_i \to T_i$ (resp. $f'_i :D'_0\cup \ldots \cup D'_{i-1} \to T'_i$) satisfying conditions (W.1-4) with associated length strictly less than $W_\theta(D'_i)+\delta$ (resp. $W_\theta(D'_0\cup \ldots \cup D'_{i-1})+\delta$). Let $e_i$ (resp. $e'_i$) denote the terminal edge of $T_i$ (resp. $T'_i$) for which the preimage of the terminal vertex is the boundary  $\partial D'_i$ (resp. $\beta_{i-1}=\partial (D'_0\cup \ldots \cup D'_{i-1})$). Suppose that $T_i$ and $T'_i$ are endowed with the metric with all edges of length $1$. In particular $e_i$ and $e'_i$ are isometric to the segment $[0,1]$. Fix $0<\epsilon<1$  and denote by $U_i\subset T_i$ the $\epsilon$-neighbourhood of the extremal vertex of $e_i \in T_i$ (resp.by $U'_i\subset T'_i$ the $\epsilon$-neighbourhood of the extremal vertex of $e'_i \in T'_i$). 
Denote by $V_i$ the closure of the union of the preimages $f_i^{-1}(U_i)$ and ${f'}_i^{-1}(U'_i)$ which is isomorphic to a sphere with three boundary components represented in Fig. 1:

\begin{itemize}
\item the curve denoted $\gamma_{i1}$ obtained as the preimage by $f_i$ of the new extremal vertex of $T_i\setminus U_i$ corresponding to the truncated edge $e_i$,
\item the curve denoted $\gamma_{i2}$ obtained as the preimage by $f'_i$ of the new extremal vertex of $T'_i\setminus U'_i$ corresponding to the truncated edge $e'_i$,
\item the boundary curve $\beta_i$.
\end{itemize}

\medskip
\begin{center}
    \includegraphics{WidthFig1.pdf} \\
    \smallskip
    \begin{small}
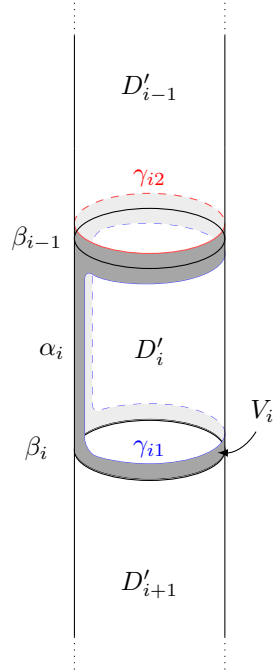
        { {\bf Fig.1. The annulus ${\mathcal A}$ near $D'_i$.} }
    \end{small}
\end{center}
\medskip

Observe that $\gamma_{i1}$ is a small deformation of  $\partial D'_i$ viewed as a curve in $D'_i$, while $\gamma_{i2}$ is a small deformation of  $\beta_{i-1} \subset D'_0 \cup \ldots \cup D'_{i-1}$. In particular
$$
\vert \gamma_{i1}\vert =\vert \partial D'_i\vert +o(\epsilon)
$$
and $$
\vert \gamma_{i2}\vert =\vert \beta_{i-1}\vert +o(\epsilon).
$$
We will define a new map from $D'_0\cup \ldots \cup D'_i$ to a trivalent tree by using the restriction of the previous maps $f_i$ and $f'_i$ on the complementary regions of $V_i$, and completing it on $V_i$ using the following map.

\begin{claim}\label{claim:2}
There exists a map $f:V_i \to Y$ where $Y$ is a tripod (a trivalent tree with only three edges) and satisfying the conditions (W.1-4) with associated length 
$$
\vert \partial D'_i \vert+o(\epsilon).
$$
\end{claim}

\begin{proof}[Proof of Claim \ref{claim:2}]
The construction of the map is straightforward and is depicted in Fig. 2.
\end{proof}

\medskip
\begin{center}
    \includegraphics[scale=0.45]{WidthFig2.pdf} \\
    \smallskip
    \begin{small}
     { {\bf Fig.2. The map $f: V_i \to Y$.} }
    \end{small}
\end{center}
\medskip

Consider the trivalent tree $T''_i$ obtained from the disjoint union of $T_i\setminus U_i$, $T'_i\setminus U'_i$ and $Y$ after identification of the terminal vertices of $T_i \setminus U_i$ and $Y$ corresponding to $\gamma_{i1}$ and the one of $T'_i \setminus U'_i$ and $Y$ corresponding to $\gamma_{i2}$.  
We then define $f''_i : D'_0 \cup \ldots \cup D'_i \to T''_i$ as follows:
\begin{itemize}
\item $f''_i(x)=f_i(x)$ if $x \in D'_i \setminus V_i$;  
\item $f''_i(x)=f'_i(x)$ if $x \in D'_0\cup \ldots \cup D'_{i-1} \setminus V_i$;
\item $f''_i(x)=f(x)$ if $x \in V_i$.
\end{itemize}
By construction we have that the length of the preimages is always less than 
$$
\max\left\{W_\theta(D'_0\cup \ldots \cup D'_{i-1})+\delta,W_\theta(D'_i)+\delta, \vert \partial D'_i \vert+o(\epsilon)\right\}.
$$
This concludes the proof by letting $\epsilon \to 0$ and $\delta \to 0$ as $W_\theta(M)\geq \vert \partial M\vert$ for any Riemannian surface $M$.
\end{proof}

\end{document}